\documentclass{amsart}

\newtheorem{thm}{Theorem}[section]
\newtheorem{prop}[thm]{Proposition}
\newtheorem{lem}[thm]{Lemma}
\newtheorem{cor}[thm]{Corollary}
\newtheorem{con}[thm]{Conjecture}

\usepackage{tikz}
\theoremstyle{definition}
\newtheorem{definition}[thm]{Definition}

\usepackage{ytableau}
\usepackage{comment}

\theoremstyle{remark}

\numberwithin{equation}{section}

\begin{document}


\title{Erd\H{o}s-Ko-Rado for Perfect Matchings}


\author{Nathan Lindzey}
\address{Department of Mathematics,
Colorado State University}
\email{lindzey@uwaterloo.ca}
\urladdr{www.cs.colostate.edu/$\sim$lindzey} 



\begin{abstract}
A \emph{perfect matching} of a complete graph $K_{2n}$ is a 1-regular subgraph that contains all the vertices.  Two perfect matchings \emph{intersect} if they share an edge.  It is known that if $\mathcal{F}$ is family of intersecting perfect matchings of $K_{2n}$, then $|\mathcal{F}|  \leq (2(n-1) - 1)!!$ and if equality holds, then $\mathcal{F} = \mathcal{F}_{ij}$ where $ \mathcal{F}_{ij}$ is the family of all perfect matchings of $K_{2n}$ that contain some fixed edge $ij$.  We give a short algebraic proof of this result, resolving a question of Godsil and Meagher.  Along the way, we show that if a family $\mathcal{F}$ is \emph{non-Hamiltonian}, that is, $m \cup m' \not \cong C_{2n}$ for any $m,m' \in \mathcal{F}$, then $|\mathcal{F}| \leq (2(n-1) - 1)!!$ and this bound is met with equality if and only if $\mathcal{F} = \mathcal{F}_{ij}$.  Our results make ample use of a somewhat understudied symmetric commutative association scheme arising from the Gelfand pair $(S_{2n},S_2 \wr S_n)$.  We give an exposition of a few new interesting objects that live in this scheme as they pertain to our results.
\end{abstract}


\maketitle




\section{Introduction}

Let $\mathcal{M}_{2n}$ be the collection of all perfect matchings of $K_{2n}$, the complete graph on an even number of vertices.  In this work, we investigate families of perfect matchings $\mathcal{F} \subseteq \mathcal{M}_{2n}$ that are \emph{intersecting}, that is, $|m \cap m'| > 0$ $\forall m,m' \in \mathcal{F}$.  In particular, we seek to characterize the largest intersecting families of perfect matchings of $K_{2n}$.  Obvious candidates are families whose members all share a fixed edge:  
\[ \mathcal{F}_{ij} := \{ m \in \mathcal{M}_{2n} : \{i,j\} \in m\} \]
where $i,j \in [2n] := \{1,2,\cdots , 2n\}$, $i \neq j$.
Such a family will be referred to as \emph{trivially intersecting}.  We give a short algebraic proof of the following result.
\begin{thm}\label{thm:ekr}
If $\mathcal{F}$ is an intersecting family of perfect matchings of $K_{2n}$, then $|\mathcal{F}|  \leq (2(n-1) - 1)!!$.  Moreover, $|\mathcal{F}| =(2(n-1) - 1)!!$ if and only if $\mathcal{F} = \mathcal{F}_{ij}$ for some $i,j \in [2n]$ such that $i \neq j$.
\end{thm}

Our result makes use of \emph{the module method}, a proof technique introduced in~\cite{NewmanPhD} to give short algebraic proofs of EKR theorems for sets and vector spaces, and used by Godsil and Meagher in~\cite{GodsilM09} to give a short algebraic proof of Theorem~\ref{thm:EKRPermutations}.
\begin{thm}\label{thm:EKRPermutations}
If $\mathcal{F}$ is an intersecting family of perfect matchings of the complete bipartite graph $K_{n,n}$, then $|\mathcal{F}|  \leq (n-1)!$.  Moreover, $|\mathcal{F}| = (n-1)!$ if and only if all members of $\mathcal{F}$ share a fixed edge.
\end{thm}
In~\cite{GodsilM09} it was asked whether the module method could be used to prove Theorem~\ref{thm:ekr}, which we settle affirmatively in this work.  Our proof is similar to Godsil and Meagher's and can be seen as the non-bipartite analogue of their result.  Many of the bipartite objects that arise in their proof and other algebraic proofs of Theorem~\ref{thm:EKRPermutations} are well-studied or have since been recognized as interesting, so it is reasonable to assume that our non-bipartite objects may also be of independent interest.  
In particular, we introduce \emph{the matching derangement graph} which can be seen as the non-bipartite analogue of the \emph{the permutation derangement graph}, a central object of several EKR results that has recently enjoyed some attention outside its EKR milieu~\cite{KuW10,KuW13,Renteln07}.  We put forth a few conjectures regarding the spectrum of the matching derangement graph that are analogues of known results of the permutation derangement graph.  
Also, the non-bipartite analogue of the Birkhoff polytope arises in our work, which had not been the subject of serious study until recently in~\cite{Rothvoss14} where it was used to show that not every LP problem with an exponential number of constraints in complexity class $P$ can be expressed as an LP with polynomially many constraints. 

Since a perfect matching of $K_{2n}$ can be seen as a $n/2$-uniform partition of $[2n]$, a combinatorial proof Theorem~\ref{thm:ekr} was first given by Meagher and Moura via the EKR theorem for intersecting families of $k$-uniform partitions~\cite{MeagherM05}. The case where $k = n/2$ arises as a special case in their proof and is the most difficult part of their result.  More recently, there has been some activity on the combinatorial front towards proving the more general \emph{full} EKR conjecture for \emph{$t$-intersecting} families of perfect matchings, that is, $\mathcal{F} \subseteq \mathcal{M}_{2n}$ such that $|m \cap m'| > t$ $\forall m,m' \in \mathcal{F}$.  We say that a family is \emph{trivially t-intersecting} if it is of the following form:
\[ \mathcal{F}_{T} := \{ m \in \mathcal{M}_{2n} : T \subseteq m\} \]
where $T$ is a collection of disjoint 2-sets of $[2n]$ of size $t$.
\begin{con}\label{con:fullekr}
If $\mathcal{F}$ is a $t$-intersecting family of perfect matchings of the complete graph $K_{2n}$, then $|\mathcal{F}|  \leq (2(n-t) - 1)!!$.  Moreover, $|\mathcal{F}| =(2(n-t) - 1)!!$ if and only if $\mathcal{F}$ is a trivially $t$-intersecting family.
\end{con}
\noindent This conjecture has resisted such combinatorial attacks, which is not too surprising as there is no known combinatorial proof of the following analogous result for perfect matchings of the complete bipartite graph $K_{n,n}$.
\begin{thm}
\cite{EllisFP11} If $\mathcal{F}$ is a $t$-intersecting family of perfect matchings of the complete bipartite graph $K_{n,n}$, then $|\mathcal{F}|  \leq (n-t)!$ for sufficiently large $n$.  Moreover, $|\mathcal{F}| = (n-t)!$ if and only if all members of $\mathcal{F}$ share a fixed set of $t$ disjoint edges.
\end{thm}
\noindent An advantage to our approach is that the cast of characters is similar to~\cite{EllisFP11} which may set the stage for an algebraic proof of Conjecture~\ref{con:fullekr} for sufficiently large $n$.

\subsubsection*{Acknowledgements}
I'd like to thank Tim Penttila for his guidance as well as David Haussler for introducing me to the work of Diaconis and Holmes on random walks over matchings some time ago.

\section{Preliminaries}

All matchings considered in this work are perfect matchings of $K_{2n}$, so henceforth we refer to a perfect matching of $K_{2n}$ simply as a \emph{matching}. Let $\mathcal{M}_{2n}$ denote the set of all matchings. A matching can be interpreted as a fixed-point-free involution of $S_{2n}$ or as a partition of $[2n]$ where each part has size two.   We shall refer to the matching $e := 1~2 | 3~4 | \cdots | 2n$-$1~2n$ as the \emph{identity matching}.   Let $H_n := \{ \sigma \in S_{2n} : \sigma e = e\}$ be the subgroup of $S_{2n}$ that stabilizes the identity matching.  It is well-known that $H_n$ is the wreath product $S_2 \wr S_n$ which is isomorphic to the hyperoctahedral group of order $2^nn!$, the group of symmetries of the $n$-hypercube.  Since matchings are in one-to-one correspondence with cosets of the quotient $S_{2n}/H_n$, it follows that $|\mathcal{M}_{2n}| = (2n-1)!! = 1 \times 3 \times 5 \times \cdots \times 2n-3 \times 2n - 1$.

For any two matchings $m,m' \in M_{2n}$, let $\Gamma(m,m')$ be the multigraph on $[2n]$ whose edge multiset is the multiset union $m \cup m'$.  Clearly $\Gamma(m,m') = \Gamma(m',m)$ and by a theorem of Berge~\cite{Berge57}, this graph is composed of disjoint cycles of even parity.  Let $k$ denote the number of disjoint cycles and let $2\lambda_i$ denote the length of an even cycle. If we order the cycles from longest to shortest and divide each of their lengths by two, we see that each graph corresponds to an integer partition $\lambda = (\lambda_1, \lambda_2, \cdots , \lambda_k) \vdash n$. For any $\lambda \vdash n$, if there are $k$ parts that all have the same size $\lambda_i$, we use $\lambda_i^k$ to denote the multiplicity. Let $d(m,m'): M \times M \mapsto \lambda(n)$ denote this map where $\lambda(n)$ is the set of all integer partitions of $n$. We shall refer to $d(m,m')$ as the \emph{cycle type of $m'$ with respect to m} (or vice versa since $d(m,m') = d(m',m)$).  If one of the arguments is the identity matching, then we say $d(e,m)$ is \emph{the cycle type of m}.    Since $\Gamma(x,y) \cong \Gamma(x',y')$ if and only if $d(x,y) = d(x',y')$, let the graph $\Gamma_{\lambda}$ be a distinct representative from the isomorphism class $\lambda \vdash n$.  Illustrations of the graphs $\Gamma_{(n)}$ and $\Gamma_{(2,1^{n-2})}$ are provided in Figure~\ref{fig:graphs} where $n = 4$.  It will be convenient to let $N := 2n-1$.

\begin{figure}\label{fig:graphs}
\begin{tikzpicture}[thick, scale=0.3]
  \foreach \x in {1,...,8}{
    \pgfmathparse{(\x)*(360 - 360/8)  + 7*360/16}
    \node[draw,circle,inner sep=0.1cm] (\x) at (\pgfmathresult:5.4cm) [ultra thick] {\textbf{\x}};
  }
     \draw[red,line width = 3] (1) -- (2);
     \draw[red,line width = 3] (3) -- (4);
     \draw[red,line width = 3] (5) -- (6);
     \draw[red,line width = 3] (7) -- (8);
     \draw[blue,dotted,line width = 3] (2) -- (3);
     \draw[blue,dotted,line width = 3](4) -- (5);
     \draw[blue,dotted,line width = 3] (6) -- (7);
     \draw[blue,dotted,line width = 3] (8) -- (1);
\end{tikzpicture}
\quad \quad \quad \quad \quad \quad 
\begin{tikzpicture}[thick, scale=0.3]
  \foreach \x in {1,...,8}{
    \pgfmathparse{(\x)*(360 - 360/8)  + 7*360/16}
    \node[draw,circle,inner sep=0.1cm] (\x) at (\pgfmathresult:5.4cm) [ultra thick] {\textbf{\x}};
  }
     \path[red,line width = 3] (1) edge[bend right] (2);
     \draw[red,line width = 3] (3) -- (4);
     \path[red,line width = 3] (5) edge[bend right] (6);
     \draw[red,line width = 3] (7) -- (8);
     \path[blue,dotted,line width = 3] (1) edge[bend left] (2);
     \draw[blue,dotted,line width = 3] (3) -- (8);
     \draw[blue,dotted,line width = 3] (4) -- (7);
     \path[blue,dotted,line width = 3] (5) edge[bend left] (6);
\end{tikzpicture}
\caption{The matching $m = 2~3|4~5|6~7|1~8$ on the left has cycle type $(n) \vdash n$ whereas the matching  $m' = 1~2|3~8|4~7|5~6$ on the right has cycle type $(2,1^{n-2}) \vdash n$ where $n = 4$.}
\end{figure}
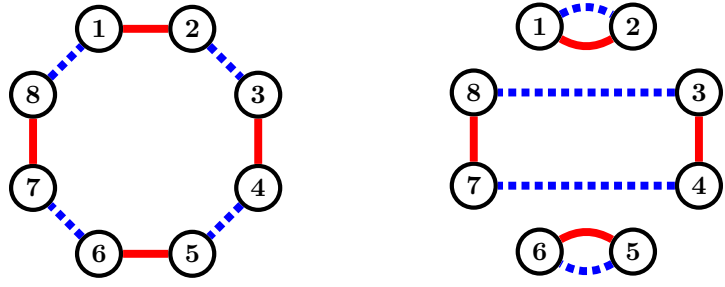

\begin{definition}\label{def:scheme}
A \emph{symmetric association scheme} is a collection of $m$ binary $n \times n$ matrices (associates) that satisfy the following axioms:
\begin{enumerate}
\item $A_i$ is symmetric.
\item $A_0 = I$ where $I$ is the identity matrix.
\item $\sum_{i=0}^m A_i = J$ where $J$ is the all-ones matrix.
\item $A_iA_j = \sum_{k=0}^m p_{ij}^k A_k =A_jA_i$
\end{enumerate}
\end{definition}
\noindent Our association scheme terminology follows~\cite{BannaiI84}. For each $\lambda \vdash n$, define the \emph{$\lambda$-associate} as the following $N!! \times N!!$ binary matrix:
\[
    (A_\lambda)_{ij} = 
\begin{cases}
    1,& \text{if } d(i,j) = \lambda\\
    0,              & \text{otherwise}
\end{cases}
\]
where $i,j \in \mathcal{M}_{2n}$.  Let $\mathcal{A}$ denote the set of all $\lambda$-associates.  It is well-known that $\mathcal{A}$ is a symmetric association scheme, so henceforth  we shall refer to $\mathcal{A}$ as \emph{the matching association scheme}. For each $\lambda \vdash n$, define the \emph{$\lambda$-sphere} (centered at $e$) to be the following set:
\[ \Omega_\lambda = \{m \in \mathcal{M}_{2n} : d(e,x) = \lambda\} \]
where $\lambda \vdash n$.  The spheres partition $\mathcal{M}_{2n}$ and it will be helpful to think of them as conjugacy classes throughout this work.  The following theorem, due to Delsarte and Hoffman, has been central to many seminal results in extremal combinatorics.
\begin{thm}\label{thm:ratio} For any weighted $k$-regular graph $\Gamma$ on $n$ vertices:
\[\alpha(\Gamma) \leq  n\frac{-\eta}{ k - \eta}\]
where $\eta$ is the smallest eigenvalue of $\Gamma$ and $\alpha(\Gamma)$ is the size of a maximum independent set in $\Gamma$.
\end{thm}
\noindent We will refer to Theorem~\ref{thm:ratio} as \emph{the ratio bound} and use it to prove a new EKR-type theorem for matchings, but the following theorem, essentially due to Delsarte, will be of central importance.
\begin{thm}\label{thm:coclique}
Let $\mathcal{A}$ be an symmetric association scheme on $n$ vertices and let $\Gamma$ be the union of some of the graphs in the scheme. If $C$ is a clique and $S$ is an independent set in $\Gamma$, then
\[|C||S| \leq n\]
If $|C||S|=n$ and $x$ and $y$ are the respective characteristic vectors
of $C$ and $S$, then
\[x^T E_j x y^T E_j y = 0~\forall j > 0.\]
\end{thm}
\noindent We shall refer to Theorem~\ref{thm:coclique} as \emph{the clique-coclique bound} and make use of one of its simple but useful corollaries.
\begin{cor}\label{cor:nonzero}
\cite{GodsilM09} Let $\Gamma$ be a union of graphs in an association scheme with the property that the Theorem~\ref{thm:coclique} holds with equality. Assume that $C$ is a maximum clique and $S$ is a maximum independent set in $\Gamma$ with characteristic vectors $x$ and $y$ respectively. If $E_j$ are the idempotents of the association scheme, then for $j > 0$ at most one of the vectors $E_j x$ and $E_j y$ is not zero.
\end{cor}

\section{The Matching Derangement Graph}

\begin{definition}
Let $\Gamma$ be the \emph{matching derangement graph} defined over $\mathcal{M}_{2n}$ such that two matchings are adjacent if and only if they are derangements of one another.
\end{definition}
\noindent 
The matching derangement graph is the analogue of the \emph{derangement graph} $\mathcal{D} = (S_n,D_n)$, that is, the normal Cayley graph defined over $S_n$ generated by the derangements (fixed-point-free permutations) $D_n \subseteq S_n$.
The number of derangements of $S_n$ is given by the following well-known recurrence.
\[ !n = D_n = (n-1)(D_{n-1} + D_{n-2})\]
where $D_0 = 1$ and $D_1 = 0$.  It is easy to see that $\mathcal{D}$ is $!n$-regular which implies that $!n$ is the largest eigenvalue of $\mathcal{D}$~\cite{GodsilRoyle}.  A unpublished result of Godsil shows that the size of a maximum clique and the chromatic number of $\mathcal{D}$ are both $n$, and it was first observed in~\cite{FranklD77} that the size of a largest independent set of $\mathcal{D}$ is $(n-1)!$.  We now prove the analogous results for the matching derangement graph.

The number of derangements of an arbitrary matching can be computed using the following lesser-known recurrence:
\[ !!n := D^M_n = 2 (n - 1) (D^M_{n - 1} +D^M_{n - 2})\]
where $D^M_0 = 1$ and $D^M_1 = 0$.  Clearly, the matching derangement graph is $!!n$-regular, hence $!!n$ is its largest eigenvalue~\cite{GodsilRoyle}.

\begin{thm}\label{clique}
The size of a maximum clique in $\Gamma$ is $2n-1$.
\end{thm}
\begin{proof}
No clique of $\Gamma$ can have more than $2n-1$ vertices, and a theorem of Lucas~\cite{Lucas92} shows that the edges of any complete graph $K_{2n}$ can be partitioned into $2n-1$ parts such that each part is a matching.
\end{proof}

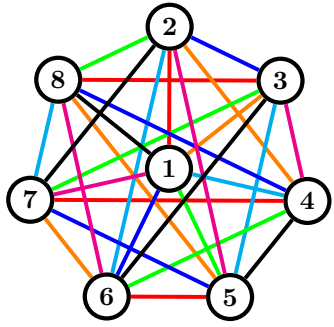
\begin{figure}
\begin{tikzpicture}[thick, scale=0.35]
  \foreach \x in {2,...,8}{
    \pgfmathparse{((\x)-1)*(360 - 360/7)  + 6*360/15.3}
    \node[draw,circle,inner sep=0.1cm] (\x) at (\pgfmathresult:5.4cm) [ultra thick] {\textbf{\x}};
  }
  \node[draw,circle,inner sep=0.1cm] (1) at (0,0) [ultra thick] {\textbf{1}};
  \draw[red,line width = 1.5] (1) -- (2);
  \draw[red,line width = 1.5] (8) -- (3);
  \draw[red,line width = 1.5] (7) -- (4);
  \draw[red,line width = 1.5] (6) -- (5);
  \draw[orange,line width = 1.5] (1) -- (3);
  \draw[orange,line width = 1.5] (2) -- (4);
  \draw[orange,line width = 1.5] (8) -- (5);
  \draw[orange,line width = 1.5] (7) -- (6);
  \draw[cyan,line width = 1.5] (1) -- (4);
  \draw[cyan,line width = 1.5] (3) -- (5);
  \draw[cyan,line width = 1.5] (2) -- (6);
  \draw[cyan,line width = 1.5] (7) -- (8);
  \draw[green,line width = 1.5] (1) -- (5);
  \draw[green,line width = 1.5] (4) -- (6);
  \draw[green,line width = 1.5] (3) -- (7);
  \draw[green,line width = 1.5] (2) -- (8);
  \draw[blue,line width = 1.5] (1) -- (6);
  \draw[blue,line width = 1.5] (5) -- (7);
  \draw[blue,line width = 1.5] (8) -- (4);
  \draw[blue,line width = 1.5] (2) -- (3);
  \draw[magenta,line width = 1.5] (1) -- (7);
  \draw[magenta,line width = 1.5] (6) -- (8);
  \draw[magenta,line width = 1.5] (5) -- (2);
  \draw[magenta,line width = 1.5] (3) -- (4);
  \draw[black,line width = 1.5] (1) -- (8);
  \draw[black,line width = 1.5] (7) -- (2);
  \draw[black,line width = 1.5] (6) -- (3);
  \draw[black,line width = 1.5] (4) -- (5);

\end{tikzpicture}
\caption{A Lucas clique of $\Gamma$ for $n = 4$.}\label{fig:onefact}
\end{figure}
\noindent Lucas in fact showed that there always exists a 1-factorization $C$ of $K_{2n}$ such that $d(m,m') = (n)$ $\forall m,m' \in C$. Such a 1-factorization will be called a \emph{Lucas clique}.  

\begin{thm}
The chromatic number of $\Gamma$ is $2n-1$.
\end{thm}
\begin{proof}
Clearly the chromatic number is greater than or equal to the clique number $2n-1$. Each member of the partition $(\mathcal{F}_{1,2}, \mathcal{F}_{1,3}, \cdots ,\mathcal{F}_{1,2n})$ is an independent set of $\Gamma$ which gives rise to a $(2n-1)$-coloring of $\Gamma$.
\end{proof}

\begin{prop}\label{scheme}
$\Gamma$ is a union of members of the association scheme $\mathcal{A}$.
\end{prop}
\begin{proof}
$\Gamma = \bigcup_{\lambda} A_\lambda$ where $\lambda$ ranges over integer partitions that have no 1-cycle.
\end{proof}

\begin{thm}
The size of a maximum independent in $\Gamma$ is $(2(n-1)-1)!!$.
\end{thm}
\begin{proof}
Any trivially intersecting family $\mathcal{F}_{ij}$ corresponds to a maximal independent set $S \subseteq \Gamma$ of size $(2(n-1)-1)!!$.  Applying Proposition~\ref{scheme} along with Theorems~\ref{clique} and~\ref{thm:coclique} gives the result.
\end{proof}

Let $X = \{X_1,\cdots,X_m\}$ be a vertex partition a graph $G$.  Then $X$ is \emph{equitable} if there exist parameters $q_{ij}$ $(1 \leq i,j \leq m)$ such that every vertex in $X_i$ is connected to precisely $q_{ij}$ vertices in $X_j$.  Let $Q=(q_{ij})$ be the \emph{quotient matrix} of $G$ with respect to $X$.

\begin{lem}\label{equit}
Let $Q$ be the quotient matrix of a graph $G$ with respect to an equitable partition $X$.  Then the eigenvalues of $Q$ correspond to eigenvalues of $G$.
\end{lem}

\begin{thm}
$!!n$ and $-\frac{!!n}{2(n-1)}$ are eigenvalues of $\Gamma$.
\end{thm}
\begin{proof}
$\Gamma$ admits an equitable partition $X = (\mathcal{F}_{ij}, \mathcal{M}_{2n} \setminus \mathcal{F}_{ij})$ with quotient matrix:
\[ \left( \begin{array}{cc}
0 & !!n  \\
\frac{!!n}{2(n-1)} & !!n - \frac{!!n}{2(n-1)} \end{array} \right)\] 
whose eigenvalues are $!!n$ and $-\frac{!!n}{2(n-1)}$, which are eigenvalues of $\Gamma$ by Lemma~\ref{equit}.
\end{proof}

Table~\ref{table:spectrum} lists the eigenvalues of the matching derangement graph for small $n$. Since $S_{2n}$ acts transitively on $\mathcal{M}_{2n}$, $\Gamma$ is vertex-transitive; however, no group acts regularly on $\mathcal{M}_{2n}$, so by Sabadussi's theorem $\Gamma$ is not a Cayley graph.  The absence of a group structure on the vertices will force us to use more general representation-theoretic techniques which we shall now develop.

\begin{table}
\caption{The spectra of $\Gamma_n$ for $n = 3,4,5,6$.  The multiplicity of the eigenvalue corresponding to $\lambda \vdash n$ is given by Theorem ~\ref{thm:eigs}.}\label{table:spectrum}
\footnotesize
\begin{tabular}{c|c|c}
$1^3$ & $21$ & $3$\\
\hline
2&   -2&   8\\
\hline
\end{tabular}\\

\bigskip
\begin{tabular}{c|c|c|c|c}
$1^4$ & $21^2$ & $2^2$ & $31$ & $4$ \\
\hline
-3 &   2&   5 &  -10&  60\\
\hline
\end{tabular}\\

\bigskip
\begin{tabular}{c|c|c|c|c|c|c}
$1^5$ & $21^3$ & $2^21$ & $31^2$ & $32$ & $41$ & $5$\\
\hline
4&   -3&   -6&  12 &  12 &  -68 &   544\\
\hline
\end{tabular}\\

\bigskip

\begin{tabular}{c|c|c|c|c|c|c|c|c|c|c}
$1^6$ & $21^4$ & $2^21^2$ & $2^3$ & $31^3$ & $321$ & $3^2$ & $41^2$ & $42$ & $51$ & $6$ \\
\hline
$-29$ & $2$ & $70$ & $10$ & $-14$ & $-14$ & $-5$ & $76$ & $82$ & $-604$ & $6040$ \\
\hline
\end{tabular}
\end{table}

\section{Finite Gelfand Pairs and their Zonal Spherical Functions}

Let $\mathbb{C}G$ be the space of complex-valued functions over a group $G$.  For any choice of $K \leq G$ there is a corresponding subalgebra $C(G,K) \leq \mathbb{C}G$ of functions that are constant on each double coset $KxK$ in $G$, that is, $C(G,K) = \{ f \in \mathbb{C}G : f(kxk') = f(x)~\forall x \in G,~\forall k,k' \in K \}$.  The theory of \emph{Gelfand pairs} provides necessary and sufficient conditions for $C(G,K)$ being commutative.
\begin{thm}\label{thm:gelfandPair}
\cite{MacDonald95} Let $K \leq G$ be a group.  Then the following are equivalent.
\begin{enumerate}
\item $(G,K)$ is a Gelfand Pair;
\item The induced representation $1_K^G = \bigoplus V_i$ (permutation representation of $G$ acting on $G/K$) is multiplicity-free;
\item The algebra $C(G,K)$ is commutative.
\end{enumerate}
\end{thm}
\noindent Let $(G,K)$ be a Gelfand pair and define $\chi_i$ to be the character of $V_i$.  The functions
\[  \omega_i(x) = \frac{1}{|K|} \sum_{k \in K} \overline{\chi_i} (xk) = \sum_{k \in K} \chi_i (x^{-1}k) \]
form an orthogonal basis for $C(G,K)$ and are called the \emph{zonal spherical functions}.  It will be helpful to think of the zonal spherical functions as the spherical analogues of characters of irreducible representations.

It is well known that $(S_{2n}, H_n)$ is a Gelfand pair, so the induced representation $1^{S_{2n}}_{H_n}$ admits the following unique decomposition into irreducible representations.
\begin{thm}
\cite{Thrall42} Let $\lambda = (\lambda_1, \lambda_2,\cdots,\lambda_k) \vdash n$ and $S^{2\lambda}$ be the Specht module of $S_{2n}$ corresponding to the partition $2\lambda := (2\lambda_1, 2\lambda_2,\cdots,2\lambda_k) \vdash 2n$. Then
\[ 1^{S_{2n}}_{H_n} = \bigoplus_{\lambda \vdash n} S^{2\lambda}.\]
\end{thm}
\noindent For any choice of $K \leq G$, a general procedure is given in~\cite{BannaiI84} for constructing a (non-commutative) association scheme whose Hecke algebra is isomorphic to $C(G,K)$.  An association scheme $\mathcal{A}$ that arises from this construction will be called a \emph{$K \backslash G / K$-association scheme}.  In such a scheme, there is a natural bijection between the associates of $\mathcal{A}$ and the double cosets and it is well-known that if $\mathcal{A}$ is a $K \backslash G / K$-association scheme, then $\mathcal{A}$ is commutative if and only if $(G,K)$ is a Gelfand pair~\cite{BannaiI84}.
The following theorem is a representation-theoretic characterization of the spectrum of any graph that arises from a $K \backslash G / K$-association scheme where $(G,K)$ is a finite Gelfand pair, which is essentially given in~\cite{BannaiI84}. 
\begin{thm}\label{thm:eigs}
Let $\Gamma = \bigcup_{j}^{\Lambda} A_{j}$ be a union of graphs in a $K \backslash G / K$-association scheme where $(G,K)$ is a Gelfand pair and $\Lambda$ is the index set of some subset of the associates.  The eigenvalue $\eta_i$ of $\Gamma$ corresponding to irreducible $i$ in the multiplicity-free decomposition of $1_K^G$ can be written as:
\[ \eta_i = \sum_{j \in \Lambda} | \Omega_j | \omega_i^{j}\]
where $\omega_i^j$ is the value of the zonal spherical function corresponding to irreducible $i$ on the double coset corresponding to $\Omega_j$. Moreover, $\eta_i$ has multiplicity $\dim \chi_i$.
\end{thm}
Theorem~\ref{thm:eigs} tells us that determining the eigenvalues of the matching derangement graph amounts to determining the sizes of $\lambda$-spheres as well as the values of the zonal spherical functions over $\lambda$-spheres.  It is well-known that when $G = H \times H$ and $K = H$, the Gelfand pair $(G,K)$ corresponds to the group representation theory of $H$, in which case Theorem~\ref{thm:eigs} gives a representation-theoretic characterization of the spectrum of any normal Cayley graph defined over $H$.  This well-known corollary of Theorem~\ref{thm:eigs} was first observed by Schur and popularized by Diaconis~\cite{Diaconis88}.

The primitive idempotents $E_\lambda$ of the matching association scheme can be computed via its character table as follows.  Let $p_\lambda(\mu)$ denote the eigenvalue of $A_\mu$ corresponding to the irreducible $\lambda$, $k_\mu = p_{(n)}(\lambda)$ be the \emph{degree} of $A_{\mu}$, and let $m_\mu$ denote the \emph{multiplicity} of $A_\mu$.  

\begin{prop}\label{prop:matchingIdempotents}
\cite{BannaiI84} The primitive idempotents of the matching association scheme $\mathcal{A}$ are defined as follows.
\[E_\lambda = \frac{m_\lambda}{N!!} \sum_{\mu \vdash n} \frac{p_\mu(\lambda)}{k_\mu}  A_\mu\]
\[(E_\lambda)_{xy} = \frac{m_{\lambda}}{N!!}  \frac{p_{d(x,y)}(\lambda)}{k_{d(x,y)}}  \]
Moreover, the entry $(E_\lambda)_{xy} = 0$ if and only if $\omega^{d(x,y)}_\lambda = 0$.
\end{prop}
\noindent Knowing the entries of the idempotents will allow us to determine $E_\lambda x \neq 0$ for certain binary vectors $x$.  The following lemma is essentially given in~\cite{MacDonald95}.  
\begin{lem}\label{lem:sphereSize}
Let $l(\lambda)$ denote the number of parts of $\lambda \vdash n$, $m_i$ denote the number of parts of $\lambda$ that equal $i$, and set $z_\lambda := \prod_{i \geq 1} i^{m_i} m_i!$.  Then the size of a $\lambda$-sphere can be computed as follows.
\[k_\lambda = |\Omega_\lambda| = \frac{|H_n|}{2^{l(\lambda)} z_\lambda}\]
\end{lem}
It is easy to see that Lemma~\ref{lem:sphereSize} is a spherical analogue of the elementary formula for determining the size of a conjugacy class $\lambda \vdash n$ of $S_n$.  Since any associate $A_\lambda \in \mathcal{A}$ is $k_\lambda$-regular, Lemma~\ref{lem:sphereSize} gives an explicit formula for computing the largest eigenvalue, $p_{(n)}(\lambda)$, of each $\lambda$-associate.  We conclude this section with two unpublished results of Diaconis and Lander that will be crucial for our main result~\cite{MacDonald95}.
\begin{lem}\label{lem:zonalId}
$\omega^{(1^n)}_\lambda = 1$
\end{lem}
\begin{lem}\label{lem:zonalCycle}
Let $\lambda \vdash n$ be a shape and $c$ be a cell of $\lambda$. Let $w(c)$ count the number of cells in $c$'s row that lie west of $c$ and $n(c)$ count the number of cells in $c$'s column that lie north of $c$. Then
\[ \omega_{\lambda}^{(n)} = \frac{1}{|H_{n-1}|} \prod_{c \in \lambda} (2w(c) - n(c))  \]
where the product excludes the cell in the upper-left corner.  Moreover, if $\lambda$ covers $2^3$, then $\omega_\lambda^{(n)} = 0$.
\end{lem}

\begin{figure}
\centering

\ytableausetup{mathmode}
\begin{ytableau}
*(yellow) \empty & *(yellow) 2 & *(yellow) 4 & *(yellow) 6 & *(yellow) 8 & 10 & 12 & \ldots  \\
*(yellow) -1 & *(yellow) 1 & 3 & 5 & 7 & 9 & 11 & \ldots \\ 
*(yellow) -2 & \mathbf{0} \\
-3 \\
-4\\
\vdots \\
\end{ytableau}

\caption{An illustration of the cell values in the product of Lemma~\ref{lem:zonalCycle}.  The colored cells compose the shape $(5,2,1) \vdash 8$ yielding $\omega_{(5,2,1)}^{(8)} = \frac{2(4)(6)(8)(-1)(1)(-2)}{2^77!} = \frac{1}{840}$.}  \label{fig:zonalFig}
\end{figure}

\section{The $2(n) \bigoplus 2(n-1,1)$ Module}
We now show that the characteristic vector $v_S$ of any maximum independent set $S$ of $\Gamma$ lives in the sum of the trivial and standard modules.  Our proof is similar to Godsil and Meagher's in the bipartite setting~\cite{GodsilM09}.
\begin{thm}\label{thm:standardModule}
If $v_S$ is the characteristic vector of a maximum independent set $S$ of $\Gamma$, then $v_S \in S^{2(n)} \bigoplus S^{2(n-1,1)}$. 
\end{thm}
\begin{proof}
For any maximum clique $C$ of $\Gamma$ define
\[ \omega_{\lambda} (C) := \sum_{c \in C} \omega_{\lambda}(c). \]
Since $\Gamma$ meets Theorem~\ref{thm:coclique} with equality, we have Corollary~\ref{cor:nonzero} at our disposal.
If there exists a maximum clique $C$ such that $\omega_\lambda(C)  \neq 0$ $\forall \lambda \neq (n-1,1) \text{ or } (n)$, then by Proposition~\ref{prop:matchingIdempotents} it follows that $E_{\lambda}v_C \neq 0$ $\forall \lambda \neq (n-1,1) \text{ or } (n)$.  By Corollary~\ref{cor:nonzero}, this would imply that $E_\lambda v_S = 0$ $\forall \lambda \neq (n-1,1) \text{ or } (n)$, giving the result.  The following shows that such a maximum clique exists.

Let $C$ be a Lucas clique that includes the identity matching $e \in C$.  Zonal spherical functions are constant on double cosets (spheres), so we write $\omega_{\lambda}(C)$ as follows:

\begin{align*} 
\omega_{\lambda}(C) &= \sum_{c \in C} \omega_{\lambda} (c)\\
&=  \omega_{\lambda}^{1^n} + 2(n-1) \omega_{\lambda}^{(n)}
\end{align*}
By Lemma~\ref{lem:zonalId}, $\omega_\lambda^{1^n} = 1$ $\forall \lambda \vdash n$, so it suffices to show such that $\omega_\lambda ^{(n)} \neq - \frac{1}{2(n-1)}$ for all $\lambda \neq (n)$ or $(n-1,1)$.  To this end, we prove that $\omega_{(n-1,1)}(C) = 0$, then show 
\[| \omega_{(n)}^{(n)} | > | \omega_{(n-1,1)}^{(n)} | > | \omega_\lambda^{(n)} | \text{~for all~} \lambda \neq (n) \text{ or } (n-1,1).\]
By Lemma~\ref{lem:zonalCycle} it follows that $\omega_{(n-1,1)}(C) = 0$ since \[\omega_{(n-1,1)}^{(n)} = \frac{-|H_{n-2}|}{|H_{n-1}|} = -\frac{1}{2(n-1)}.\]  
It suffices to show that $|H_{n-2}|$ is the largest value that the numerator of Lemma~\ref{lem:zonalCycle} can assume and this occurs only when $\lambda = (n-1,1)$. Lemma~\ref{lem:zonalCycle} states that the only $\lambda \vdash n$ that do not evaluate to zero must be of the form $(n-k,1^k)$ where $0 \leq k< n$ or $(n-j,j-k,1^k)$ where $0 \leq k < j < n$.  

For any shape $\lambda = (n-k,1^k)$ where $k > \frac{n}{2}$, we have $|\omega_\lambda^{(n)}| < |\omega_{\lambda'}^{(n)}|$ where $\lambda' \vdash n$ is the transpose of $\lambda$. It is also the case that $|\omega_{(n-k, 1^k)}^{(n)}| < |\omega_{(n-k+1, 1^{k-1})}^{(n)}|$ where $1 \leq k \leq \frac{n}{2}$. It follows that $| \omega_{(n-1,1)}^{(n)} | > | \omega_\lambda^{(n)} |$ holds for all $\lambda = (n-k,1^k)$, $k > 1$.

Let $\lambda = (n-j, j-k,1^k)$ where $2 < k < j < n$ and let $\mu = \lambda \setminus \lambda_1 $ be shape obtained by removing the first row. For $\mu = (j-k,1^k)$ where $k > \frac{j}{2}$, using similar reasoning, we have $|\omega_{(\lambda_1,\mu)}^{(j)}| < |\omega_{(\lambda_1,\mu')}^{(j)}|$. It is also true that $|\omega_{(\lambda_1,j-k, 1^k)}^{(n)}| < |\omega_{(\lambda_1,j-k+1, 1^{k-1})}^{(n)}|$ where $1 \leq k < \frac{j}{2}$.  
For the case where $1 \leq k \leq 2$, it is easy to see that removing the bottom left cell of $\lambda$ and placing it in the upper right hand corner always gives a new shape with a larger character sum, hence $| \omega_{(n-1,1)}^{(n)} | > | \omega_\lambda^{(n)} |$ for all valid shapes of the form $(n-j,j-k,1^k)$, which completes the proof.
\end{proof}

\begin{cor}
The minimum eigenvalue of $A_{(n)} \in \mathcal{A}$ is $p_{(n-1,1)}((n)) = -|H_{n-2}|$.
\end{cor}
\begin{proof}
Since $p_\lambda((n)) = |\Omega_{(n)}| \omega_{\lambda}^{(n)}$ and $p_{(n)}((n)) = |\Omega_{(n)}|\omega_{(n)}^{(n)} = |\Omega_{(n)}|$ is always positive, it follows that $p_{(n-1,1)}((n))$ is the unique least eigenvalue of $A_{(n)}$.  Moreover, 
\begin{align*}
p_{(n-1,1)}((n)) &= |\Omega_{(n)}| \omega_{(n-1,1)}^{(n)}\\
&= |H_{n-1}| -\frac{|H_{n-2}|}{|H_{n-1}|} \quad\quad \text{(Lemma~\ref{lem:sphereSize})}\\
&= - |H_{n-2}|
\end{align*}
\end{proof}
\noindent The corollary above along with Theorem~\ref{thm:ratio} implies the following EKR-type result.
\begin{thm}\label{thm:cycleEKR}
Let $\mathcal{F}$ be a family of matchings such that for any two members $x,y \in \mathcal{F}$, $x \cup y$ is disconnected.  Then $|\mathcal{F}| \leq (2(n-1)-1)!!$.  This bound is tight.
\end{thm}
\begin{proof}
By the ratio bound, we have 
\begin{align*}
|\mathcal{F}| \leq (2n-1)!! \frac{|H_{n-2}|}{|H_{n-1}| + |H_{n-2}|} &= \frac{(2n-1)!}{|H_{n-2}|(2n-1)}\\
 &= (2(n-1)-1)!!
\end{align*}
Any family of the form $\mathcal{F}_{ij}$ is maximum independent set of $\Gamma$ of size $(2(n-1)-1)!!$. Since $A_{(n)}$ is a subgraph of $\Gamma$, $\mathcal{F}_{ij}$ must also be a maximum independent set of $A_{(n)}$, hence the bound is tight.
\end{proof}
\noindent This is somewhat surprising since it tells us that a maximum independent set of the matching derangement graph does not increase in size even if we remove all edges except those that belong to the $(n)$-associate.  A similar result has been observed in the conjugacy class association scheme of $S_n$~\cite{AhmadiPhD}.  Later, we will observe the stronger result that a non-Hamiltonian family $\mathcal{F}$ is largest if and only if $\mathcal{F}$ is a trivially intersecting family. 

\section{The Perfect Matching Polytope of $K_{2n}$}

We now complete the final part of the proof of the Erd\H{o}s-Ko-Rado theorem for intersecting families of perfect matchings.  The proof of the theorem below follows a polyhedral method of Godsil and Meagher that can be found in~\cite{RooneyPhD}.

\begin{thm}\label{thm:polytope}
If $\mathcal{F}$ is a maximum intersecting family of matchings, then $\mathcal{F}$ is a trivially intersecting family of matchings.
\end{thm}

Let $G = (V,E)$ be a graph on an even number of vertices.  For every perfect matching $m$ of $G$ there is an associated incidence vector $x$ defined such that $x_i = 1$ if $i \in m$; otherwise, $x_i = 0$.  The convex hull of the incidence vectors of perfect matchings of $G$ is denoted $PM(G)$ and is known as \emph{the perfect matching polytope} of $G$.  The following terminology is necessary in order to define $PM(K_{2n})$ as the solution set of a system of linear inequalities. 

For any $S \subseteq V$ or vertex $S \in V$, let $\delta(S)$ denote the set of edges with exactly one endpoint incident to $S$. An \emph{odd cut} $C$ of a graph $G$ is a set of edges of the form $\delta(S)$ where $S$ is a non-empty proper subset of $V$ of odd size, $|V \setminus S| > 1$, and $|S| > 1$. Let $\mathcal{C}$ be the set of all odd cuts of a graph. If $G$ is a connected graph such that every edge belongs to some perfect matching, then $x \in PM(G)$ if and only if it satisfies the set of linear inequalities below~\cite{Edmonds65}.
\begin{align}
x_i &\geq 0~~\forall i \in E\\
\sum_{i \in \delta(v)} x_i &= 1~~\forall v \in V\\
\sum_{i \in C} x_i &\geq 1~~\forall C \in \mathcal{C}
\end{align}
A special case of this polytope that has been thoroughly studied is \emph{Birkhoff polytope}, otherwise known as \emph{the perfect matching polytope of $K_{n,n}$}.  We refer the reader to~\cite{Pak00} for a succinct but thorough exposition of this polytope as well as its import to several branches of mathematics.  

Let $r(G)$ denote the maximum number of perfect matchings in graph $G$ whose incidence vectors are linearly independent over $\mathbb{R}$.  In~\cite{EdmondsPL82}, it is observed that
\[r(G) = 1 + \dim PM(G)\]
\begin{thm}
\cite{EdmondsPL82} Let $G = (V,E)$ be a connected graph such that every edge belongs to some perfect matching.  Then $\dim PM(G) = |E| - |V| + 1 - \beta$ where $\beta$ is the number of bricks in the brick decomposition of $G$.
\end{thm}
\noindent Since the size of the brick decomposition of a non-trivial clique on an even number of vertices is 1~\cite{EdmondsPL82}, with a little representation theory, we have the following corollary.
\begin{cor}\label{cor:dimpoly}
\[r(K_{2n}) = 1 + \dim M(K_{2n}) = \binom{2n}{2} - 2n + 1 = \dim S^{2(n)} \bigoplus S^{2(n-1,1)}\]
\end{cor} 	
Using highly non-trivial graph theory, the facet-inducing inequalities of $PM(G)$ are also characterized in~\cite{EdmondsPL82}.  Their characterization of the facet-inducing inequalities is much too involved to be stated in full, so we refer the interested reader to Theorem 6.3 of~\cite{EdmondsPL82} for the full graph-theoretical characterization of the facets of $PM(G)$.  Fortunately, when $G = K_{2n}$ the situation is drastically simplified, leading to the following straightforward corollary of Theorem 6.2 of~\cite{EdmondsPL82}.
\begin{cor}\label{cor:facetpoly}
The inequalities \emph{(6.1)} and \emph{(6.3)} are precisely the facet-inducing inequalities of $PM(K_{2n})$. Moreover, each facet of $PM(K_{2n})$ can be written in strictly one of the following forms.
\begin{align*}
F_e &= \{ x \in PM(K_{2n}) : x_e = 0\}\\
F_C &= \{ x \in PM(K_{2n}) : \sum_{e \in C} x_e = 1 \} 
\end{align*}
\end{cor}
\noindent Finally, recall that any face of a polytope $P$ can be expressed as the intersection of $P$ with a hyperplane $H$.  In particular, if we let $h \in R^n$, then for each $a \in \mathbb{R}$,
\[H_a = \{x \in \mathbb{R}^n : h^Tx=a\}\] 
form hyperplanes that partition $\mathbb{R}^n$. If $P$ is a polytope, then there is some $a \in \mathbb{R}$ such that $P \cap H_a \neq \emptyset$.  It follows that by finding the maximum and minimum values of $a$ such that $P \cap H_a \neq \emptyset$, we find \emph{parallel} faces of $P$.
\begin{lem}\label{lem:parallelFaces}
Let $P$ be the convex hull of the rows of a matrix $M$, then $Mh=z$ and
\begin{align*}
F_{\min} &= \{ x \in P : h^Tx = \min(z)\}\\
F_{\max} &=\{ x \in P : h^Tx = \max(z)\}
\end{align*}
are parallel faces of $P$ where $\min(z)$ and $\max(z)$ denotes the minimum and maximum value of any component of $z$.  Moreover, if $z$ is a 0/1 vector, then $F_{\min}$ and $F_{\max}$ partition the vertices of $P$.
\end{lem}

\subsubsection*{Proof of Theorem~\ref{thm:polytope}} Let $M$ be a $N!! \times \binom{2n}{2}$ binary matrix whose columns are the characteristic vectors $v_{ij}$ of the trivially intersecting families $\mathcal{F}_{ij}$.  By Theorem~\ref{thm:standardModule}, the columns of $M$ live in the sum of the trivial and standard modules, and by Corollary~\ref{cor:dimpoly}, the columns of $M$ span the sum of the trivial and standard modules.  It follows that for any maximum independent set $Z$ of $\Gamma$, its characteristic vector can be written as $z = Mh$.  Since $z$ is a 0/1 vector, it follows by Lemma~\ref{lem:parallelFaces} that $F_0$ and $F_1$ are parallel faces that partition the vertices of $PM(K_{2n})$.  Every face is contained in a facet, so there is some facet $F$ such that $F_0 \subseteq F$.  The following shows that if $z$ is the characteristic vector of a maximum independent set, then $F$ is of the form $F_e$ for some $e \in E$.

Suppose that $F_0 \subseteq F = F_C$ for some odd cut $C$. Then $\sum_{e \in C} x_e = 1$ for every vertex $x \in F_0$, or equivalently, each vertex of $F_0$ uses precisely one edge of $C$. Since $F_0$ and $F_1$ partition the vertices of $PM(K_{2n})$, we have $\sum_{e \in C} x_e > 1$ for all vertices $x \in F_1$.  Let $S$ be an odd set induced by $C$ and define $s := |S|= 2k+1$ for some $k > 0$. A simple counting argument reveals that there are $(2n-1)!! - s!!(2n-s)!!$ vertices of $F_1$, each of which corresponds to a non-zero component of $z$.  This gives us a contradiction as $z$ now has too many non-zero components to be the characteristic vector of a maximum independent set, that is,
\begin{align*}
(2n-1)!! - s!!(2n-s)!! &> (2(n-1)-1)!!\\
\frac{(2n-1)!!}{s!!(2n-s)!!} - 1 &> \frac{(2(n-1)-1)!!}{s!!(2n-s)!!}
\end{align*}
which by induction is true for $n > 2$ and any valid choice of $s$.

It follows that $F_0 \subseteq F_e$ for some $e = \{i,j\}$.  Since $x_e = 0$ $\forall x \in F_0$, we have $x_e = 1$ $\forall x \in F_1$.  The vertices of $F_1$ form the support of $z$, and since there are precisely $(2(n-1)-1)!!$ matchings of the form $x_e = 1$, it must be that $z = v_{ij}$ for some $i,j \in [2n]$, which completes the proof of the main result.

Notice that the proof of the main result only depended on properties of the $(n)$-associate $A_{(n)} \in \mathcal{A}$, so it is not difficult to see that the proofs of Theorems~\ref{thm:standardModule} and~\ref{thm:polytope} remain true when the graph in question is the $(n)$-associate rather than $\Gamma$.  Indeed, a Lucas clique of $\Gamma$ is also a maximum clique of $A_{(n)}$ and we only needed zonal spherical function values over the $(n)$-sphere.  
\begin{thm}
Theorem~\ref{thm:cycleEKR} is met with equality if and only if $\mathcal{F}$ is a trivially intersecting family of matchings. 
\end{thm}

\section{Open Questions}

Let $\eta$ be the minimum eigenvalue of the permutation derangement graph $\mathcal{D}$. By Theorem~\ref{thm:ratio}, we have $\eta \leq -{!n}/{(n-1)}$, which prompted Ku to conjecture that $\eta = -{!n}/{(n-1)}$.  Renteln later proved this conjecture using non-trivial symmetric function theory~\cite{Renteln07}.
For $n < 8$, it has been verified in \texttt{GAP}~\cite{GAP} that $\Gamma$ meets the Theorem~\ref{thm:ratio} with equality which motivates the following conjecture.
\begin{con}
The least eigenvalue of the matching derangement graph is $-\frac{!!n}{2(n-1)}$.  
\end{con}
\noindent In~\cite{KuW10}, Ku and Wales show that the spectrum of $\mathcal{D}$ possesses the so-called \emph{alternating sign property}.  For any $\lambda \vdash n$, define the \emph{depth} of $\lambda$ to be $d(\lambda) := n - \lambda_1$, i.e., the number of cells under the first row of $\lambda$. For $n < 8$, the matching derangement graph too possesses this property and so the following is likely true.
\begin{con}
$\Gamma$ has the alternating sign property, that is, for any $\lambda \vdash n$
\[ \emph{sign}(\eta_{\lambda}) = (-1)^{d(\lambda)}\]
where $\eta_{\lambda}$ is an eigenvalue of $\Gamma$.
\end{con}

For future work, one would expect the next step to be toward an algebraic proof of the EKR theorem for $n/k$-uniform partitions of $[kn]$ for $k >2$; however, the following theorem of Saxl is something of an albatross.
\begin{thm}\label{thm:gelfandClassification}
Let $n > 18$ and $H \leq S_n$. If $1^{S_n}_H$ is multiplicity free, then one of the following holds:
\begin{enumerate}
\item $A_{n-k} \times A_k \leq H \leq S_{n-k} \times S_k$ for some $k$ with $0 \leq k < n/2$;
\item $n = 2k$ and $A_k \times A_k < H \leq S_k \wr S_2$;
\item $n = 2k$ and $H \leq S_2 \wr S_k$ with $[S_2 \wr S_k : H] \leq 4$;
\item $n = 2k + 1$ and $H$ fixes a point of $[1, n]$ and is one of the subgroups in
\emph{(2)} or \emph{(3)} on the rest of $[1, n]$;
\item $A_{n-k} \times G_k \leq H \leq S_{n-k} \times G_k$ where $k = 5,6,9$ and $G_k$ is $AGL(1, 5),GL(2, 5)$ or $P\Gamma L(2, 8)$ respectively.
\end{enumerate}
\end{thm}
Theorem~\ref{thm:gelfandClassification} has been further refined by Godsil and Meagher, who in~\cite{GodsilM10} provide a complete list of the multiplicity-free permutation representations of $S_n$ for all $n$.  By Theorem~\ref{thm:gelfandPair} this gives a complete list of the infinite families of finite Gelfand pairs of the form $(S_n,K)$ where $K \leq S_n$.  For $k > 2$, we see that $1^{S_{kn}}_{S_k \wr S_n}$ is not multiplicity-free, hence the $(S_k \wr S_n) \backslash S_{kn} / (S_k \wr S_n)$-association scheme is typically not commutative.  In general, determining the multiplicities of the irreducibles of $1^{G}_{K}$ for arbitrary $K \leq G$ is difficult, even if $K$ is restricted to be of the form $S_k \wr S_n$.  Progress in this direction is related to Folkes' conjecture; however, most of the techniques in this area do not seek upper bounds on multiplicities, which would be beneficial in our setting. 

Recall that the module method relies on a union of members of an association scheme $\mathcal{A}$ meeting the clique-coclique bound with equality.  From~\cite{GodsilM09} and this work, we know that if $\Gamma$ is the union of fixed-point-free associates of the conjugacy class association scheme of $S_n$ or the matching association scheme, then the clique-coclique bound is always met with equality.  This is because for every order there always exists degenerate finite geometries, namely latin squares and 1-factorizations of $K_{2n}$, which are in one-to-one correspondence with cliques of size $n$ and $2n-1$ respectively of $\Gamma$.  If $\Gamma$ is a union of associates that have at most a single fixed point, then maximum cliques are cryptomorphic to unwieldy finite geometries such as projective planes and abstract hyperovals that do not exist for all orders~\cite{Polster97}.  The nonexistence of these geometric objects for certain orders implies that there is no clique in $\Gamma$ large enough to meet the clique-coclique bound with equality.  Unfortunately, knowing precisely when this happens is equivalent to the intractable problems of classifying projective planes and abstract hyperovals.  

Finally, we note that the method used by Ellis, Friedgut, and Pilpel in~\cite{EllisFP11} ignores clique structure, making it better suited for solving $t$-intersecting problems with the tradeoff that the results hold only for sufficiently large $n$.  Plausible future work would be to use their method over the Gelfand pair $(S_{2n}, S_2 \wr S_n)$ to give a proof of the full version of the Erd\H{o}s-Ko-Rado theorem for $t$-intersecting families of perfect matchings.

\bibliographystyle{plain}
\bibliography{../../master}

\end{document}